\documentclass[11pt]{amsart}
\usepackage{float}
\usepackage[usenames]{color}
\usepackage{graphicx}
\usepackage{amssymb}
\usepackage{amscd}
\theoremstyle{plain}
\newtheorem{thm}{Theorem}[section]
\newtheorem{lemma}[thm]{Lemma}
\newtheorem{cor}[thm]{Corollary}
\newtheorem{conj}[thm]{Conjecture}
\newtheorem{prop}[thm]{Proposition}

\theoremstyle{definition}
\newtheorem{definition}[thm]{Definition}
\newtheorem{remark}[thm]{Remark}

\newtheorem{alg}[thm]{Algorithm}

\newtheorem{example}[thm]{Example}

\def\dim{\mathop{\hbox {dim}}\nolimits}

\def\det{\mathop{\hbox {det}}\nolimits}

\def\im{\mathop{\hbox {Im}}\nolimits}

\def\ker{\mathop{\hbox{Ker}}\nolimits}

\def\aq{A_{\mathfrak{q}}}

\newcommand{\fra}{\mathfrak{a}}

\newcommand{\frg}{\mathfrak{g}}
\newcommand{\frh}{\mathfrak{h}}

\newcommand{\frk}{\mathfrak{k}}
\newcommand{\frl}{\mathfrak{l}}

\newcommand{\frp}{\mathfrak{p}}
\newcommand{\frq}{\mathfrak{q}}

\newcommand{\frs}{\mathfrak{s}}
\newcommand{\frt}{\mathfrak{t}}
\newcommand{\fru}{\mathfrak{u}}

\newcommand{\bbC}{\mathbb{C}}

\newcommand{\bbR}{\mathbb{R}}

\newcommand{\caL}{\mathcal{L}}

\newcommand{\caR}{\mathcal{R}}

\newcommand{\be}{\begin {equation}}
\newcommand{\ee}{\end {equation}}
\newcommand{\bp}{\begin {proof}}
\newcommand{\ep}{\end {proof}}

\begin{document}

\title[On the Dirac series of $U(p,q)$]{On the Dirac series of  $U(p,q)$}
\author{Chao-Ping Dong}
\author{Kayue Daniel Wong}

\address[Dong]{School of Mathematical Sciences, Soochow University, Suzhou 215006,
P.~R.~China}
\email{chaopindong@163.com}

\address[Wong]{School of Science and Engineering, The Chinese University of Hong Kong, Shenzhen,
Guangdong 518172, P. R. China}
\email{kayue.wong@gmail.com}

\abstract{This paper computes the Dirac index of all the weakly fair $\aq(\lambda)$ modules of $U(p, q)$. We find counter-examples to a conjecture of Vogan on the unitary dual of $U(p, q)$, which was phrased by Trapa in 2001. However, we still believe that any irreducible unitary representation of $U(p, q)$ with non-zero Dirac cohomology must be a weakly fair $\aq(\lambda)$ module.} 
 \endabstract

\subjclass[2010]{Primary 22E46.}

\keywords{Dirac cohomology, Dirac index, scattered representation, unitary representation.}

\maketitle
\section{Introduction}
Let $G$ be a connected linear Lie group with Cartan involution $\theta$. Assume that $K:=G^\theta$ is a maximal compact subgroup of $G$. Let $T$ be a maximal torus of $K$.
 Let $\frg_0={\rm Lie}(G)$,  $\frk_0={\rm Lie}(K)$, and $\frt_{0}={\rm Lie}(T)$.  Let
 $$
 \frg_0=\frk_0\oplus\frp_0
 $$
 be the Cartan decomposition on the Lie algebra level.
 Let $\fra_{0}$ be the centralizer of $\frt_{0}$ in $\frp_0$, and put $A=\exp(\fra_{0})$. Then $H=TA$ is the unique $\theta$-stable Cartan subgroup of $G$ which is maximally compact. On the Lie algebra level,
 $$
 \frh_{0}:=\frt_{0}\oplus\fra_{0}
 $$
 is called the fundamental Cartan subalgebra of $\frg_0$. Put $\frt_{\bbR}=i\frt_0$, $\frt_{\bbR}^*=i\frt_0^*$ and
 $$
 \frh_{\bbR}^*=i\frt_0^* \oplus \fra_0^*.
 $$
As usual, we drop the subscripts to stand for the complexified Lie algebras. That is, $\frg$ means $\frg_0\otimes_{\bbR} \bbC$ and so on. We fix a non-degenerate invariant symmetric bilinear form $B$ on $\frg$. Then $\frk$ and $\frp$ are orthogonal to each other under $B$.

Fix an orthonormal basis $\{Z_1, \dots, Z_n\}$ of $\frp_0$ with respect to the inner product on $\frp_0$ induced by $B$. Let $U(\frg)$ be the universal enveloping algebra, and let $C(\frp)$ be the Clifford algebra. As introduced by Parthasarathy \cite{P1},
the \emph{Dirac operator} is defined as
\begin{equation}\label{Dirac-operator}
D:=\sum_{i=1}^n Z_i\otimes Z_i\in U(\frg)\otimes C(\frp).
\end{equation}
It is easy to check that $D$ does not depend on the choice of the orthonormal basis $\{Z_i\}_{i=1}^n$. Writing out $D^2$ carefully will lead one to Parthasarathy's Dirac operator inequality \cite{P2}, which is very effective in non-unitarity test. See \eqref{Dirac-inequality}.

Let ${\rm Spin}_G$ be a spin module for the Clifford algebra $C(\frp)$.
For any $(\frg, K)$-module $\pi$, the Dirac operator $D$ acts on $\pi\otimes {\rm Spin}_G$, and the \emph{Dirac cohomology} defined by Vogan \cite{Vog97} is the following $\widetilde{K}$-module:
\begin{equation}\label{Dirac-cohomology}
H_D(\pi):=\ker D/(\ker D\cap\im D).
\end{equation}
Here $\widetilde{K}$ is the spin double cover of $K$. That is,
$$
\widetilde{K}:=\{ (k, s) \in K \times {\rm Spin}(\frp_0)\mid {\rm Ad}(k)=p(s) \},
$$
where ${\rm Ad}: K\to SO(\frp_0)$ is the adjoint map, and $p: {\rm Spin}(\frp_0) \to SO(\frp_0)$ is the universal covering map. Fix a positive root system $\Delta^+(\frk, \frt)$, and denote the half sum of roots in $\Delta^+(\frk, \frt)$ by $\rho_c$. We will use $E_{\mu}$ to denote the $\frk$-type (that is, an irreducible representation of $\frk$) with highest weight $\mu$. Abuse the notation a bit, $E_{\mu}$ will also stand for the $K$-type as well as the $\widetilde{K}$-type with highest weight $\mu$. Fix a positive root system $\Delta^+(\frg, \frt)$ containing $\Delta^+(\frk, \frt)$, and denote the half sum of roots in $\Delta^+(\frg, \frt)$ by $\rho$.

One original motivation of introducing Dirac cohomology is that this new invariant should sharpen Parthasarathy's Dirac operator inequality, and thus help us to understand the \emph{unitary dual} $\widehat{G}$---the set of equivalence classes of irreducible unitary $(\frg ,K)$-modules. This turned out to be the case. Indeed, the following Vogan conjecture proved by Huang and Pand\v zi\'c \cite{HP} says that Dirac cohomology, whenever non-zero, refines the infinitesimal character of $\pi$. Moreover, if one specializes $\pi$ to be unitary, then it is not hard to extend Parthasarathy's Dirac operator inequality, see Theorem 3.5.2 of \cite{HP2}.

\begin{thm}\emph{(\textbf{Huang-Pand\v zi\'c} \cite{HP})}\label{thm-HP}
Let $\pi$ be any irreducible $(\frg, K)$-module with infinitesimal character $\Lambda\in\frh^*$. Assume that $H_D(\pi)$ is non-zero, and that $E_{\gamma}$ is contained in $H_D(X)$. Then $\Lambda$ is conjugate to $\gamma+\rho_c$ by some element in the Weyl group $W(\frg, \frh)$.
\end{thm}

Many interesting representations such as the discrete series, and some $\aq(\lambda)$-modules \cite{HKP} (see Section \ref{sec-coho-ind}) turned out to have non-zero Dirac cohomology. Let us collect all the members of $\widehat{G}$ with non-zero Dirac cohomology as $\widehat{G}^d$, and call them the \emph{Dirac series} of $G$ as coined by Huang. This paper aims to study the Dirac series of $U(p,q)$. The recent research announcement \cite{BP19} suggests that the study here may be helpful for the theory of automorphic forms. One foundational tool for us is Theorem \ref{thm-HP}, the other one is Trapa's 2001 paper \cite{T1} which suggests that weakly fair $\aq(\lambda)$-modules (see \eqref{Aqlambda-weakly-fair}) should play a very important role in the unitary dual of $U(p, q)$. In particular, the following inspiring conjecture was stated there.

\begin{conj}\label{conj-Vogan} \emph{(\textbf{Vogan's Conjecture on $\widehat{U(p, q)}$})}\, The weakly fair $A_{\frq}(\lambda)$ modules exhaust the irreducible unitary $(\frg, K)$-modules for $U(p, q)$ whose infinitesimal character is a weight translate of $\rho$.
\end{conj}

Take any irreducible unitary $\aq(\lambda)$ module $\pi$ of $U(p, q)$, we obtain the necessary conditions \eqref{hpcondition} for $H_D(\pi)$ to be non-zero in Lemma \ref{lemma-Aqlambda-nonzeroDirac} via Theorem \ref{thm-HP}. When $\pi$ is further assumed to be weakly fair, our main result Theorem \ref{thm-DI-SUpq-Aqlambda} says that \eqref{hpcondition} plus with \eqref{eq-inequality} are actually sufficient.
More precisely, Theorem \ref{thm-DI-SUpq-Aqlambda} computes the \emph{Dirac index} (see Section \ref{sec-DI}) of  all weakly fair $\aq(\lambda)$ modules of $U(p, q)$. It says that the Dirac index never vanishes whenever \eqref{hpcondition} and \eqref{eq-inequality} are both satisfied. Hence the Dirac cohomology never vanishes in such cases as well. Earlier, Barbasch and Pand\v zi\'c have studied the Dirac cohomology of some unipotent representations of $U(p, q)$ in \cite{BP15}. Theorem 5.3 there now fits nicely into the current setting as the special case described in Example \ref{exam-BP}.

In view of Lemma \ref{lemma-Aqlambda-nonzeroDirac}, there would be no gap from Theorem \ref{thm-DI-SUpq-Aqlambda} to the complete classification of $\widehat{U(p, q)}^d$ if Conjecture \ref{conj-Vogan} holds. Unfortunately, we find counter examples to Conjecture \ref{conj-Vogan} on $U(p, p)$. See Section \ref{sec-Vogan-conj}. Fortunately, these representations have zero Dirac cohomology. Thus they are not members of $\widehat{U(p, q)}^d$, and will not bother us. Moreover, using the algorithm in \cite{D17} and the software \texttt{atlas} \cite{ALTV,At}, we have carried out calculations on small rank groups up to $U(5, 5)$: there is no gap on each group. Thus we would like to make the following.

\begin{conj} \label{conj-diracseries} Any Dirac series of $U(p, q)$ must be a weakly fair $A_{\frq}(\lambda)$ module satisfying \eqref{hpcondition} of Lemma \ref{lemma-Aqlambda-nonzeroDirac}.
\end{conj}

The paper is organized as follows: Section 2 recalls the spin module and the unitarily small convex hull. Section 3 recalls cohomologically induced modules, in particular, $\aq(\lambda)$ modules in various ranges. Dirac index will be visited in Section 4.  We study Dirac indices of weakly fair $\aq(\lambda)$ modules of $U(p, q)$ in Section 5. Finally, we investigate Conjecture \ref{conj-Vogan} in Section 6. 

\section{Spin module and the u-small convex hull}
We assume that $G$ is simple for convenience, and adopt the notations from the introduction. In this section, we will collect materials pertaining to the spin module and the unitarily small convex hull, which are key ingredients in the study of Dirac cohomology and Dirac index.

Fix a positive root system $\Delta^+(\frk, \frt)$. Choose a positive root system $\Delta^+(\frg, \frt)$ which contains $\Delta^+(\frk, \frt)$. Denote the half sum of roots in $\Delta^+(\frg, \frt)$ (resp., $\Delta^+(\frk, \frt)$) by $\rho$ (resp., $\rho_c$). Put \begin{equation}\label{rhon}
\rho_n=\rho-\rho_c.
\end{equation}
Let $C_{\frg}(\frt_{\bbR}^*)$ (resp., $C_{\frk}(\frt_{\bbR}^*)$) be the dominant Weyl chamber corresponding to $\Delta^+(\frg, \frt)$ (resp., $\Delta^+(\frk, \frt)$). Put
\begin{equation}\label{Wgtone}
W(\frg,\frt)^1:=\{w\in W(\frg,\frt)\mid w  C_{\frg}(\frt_{\bbR}^*) \subseteq C_{\frk}(\frt_{\bbR}^*)\}.
\end{equation}
Then the multiplication map is a bijection from $W(\frk, \frt) \times W(\frg, \frt)^1$ onto $W(\frg, \frt)$. Put
$$
s=\frac{\# W(\frg, \frt)}{\# W(\frk, \frt)}.
$$
Let us enumerate the set  $W(\frg, \frt)^1$ as
\begin{equation}\label{Wgtone-label}
\{w^{(0)}=e, w^{(1)}, \dots, w^{(s-1)}\}.
\end{equation}
Then $w^{(j)}\Delta^+(\frg, \frt)$, $0\leq j\leq s-1$, are exactly all the  positive root systems of $\Delta(\frg, \frt)$ containing $\Delta^+(\frk, \frt)$. Denote by $\rho^{(j)}$ the half sum of roots in $w^{(j)}\Delta^+(\frg, \frt)$, and put
\begin{equation}\label{rhonj}
\rho_n^{(j)}=\rho^{(j)}-\rho_c.
\end{equation}
Note that $\rho^{(0)}=\rho$ and $\rho_n^{(0)}=\rho_n$.

Let us recall \cite[Lemma 9.3.2]{W}.

\begin{lemma}\label{lemma-spin-G}
Let $E_{\mu}$ be the $\frk$-type with highest weight $\mu$. Then
$$
{\rm Spin}_{G} \cong \bigoplus_{j=0}^{s-1} 2^{[l_0/2]} E_{\rho_n^{(j)}},
$$
where $l_0=\dim_{\bbC} \fra$ and $m E_{\mu}$ stands for a direct sum of $m$ copies of $E_{\mu}$.
\end{lemma}

In \cite{D}, the spin norm of the $\frk$-type $\mu$ is defined as
\begin{equation}
\|\mu\|_{\rm spin}:=\min_{0\leq j\leq s-1} \|\{\mu- \rho_n^{(j)}\}+\rho_c\|.
\end{equation}
Here the norm $\|\cdot\|$ is induced by the restriction of $B$ to $\frt_{\bbR}$, and $\{\mu- \rho_n^{(j)}\}$ denotes the unique weight in the dominant Weyl chamber $C_{\frk}(\frt_{\bbR}^*)$ which is conjugate to $\mu- \rho_n^{(j)}$ under the action of the Weyl group $W(\frk, \frt)$. Parthasarathy's Dirac operator inequality can be encapsulated as follows: for any irreducible unitary $(\frg, K)$-module $\pi$ with infinitesimal character $\Lambda$, we have that
\begin{equation}\label{Dirac-inequality}
\min_{E_{\mu}}  \|\mu\|_{\rm spin} \geq \|\Lambda\|,
\end{equation}
where $E_{\mu}$ runs over all the $K$-types in $\pi$. Then Theorem 3.5.2 of \cite{HP2} extends Parthasarathy's Dirac operator inequality in the sense that equality in \eqref{Dirac-inequality} holds if and only if $H_D(\pi)$ is non-zero.

In the current setting, the convex hull generated by the points $\{2w\rho_n\mid w\in W(\frg, \frt)\}$ is the unitarily small polyhedron introduced by Salamanca-Riba and Vogan \cite{SV}. Its vertices within $C_{\frk}(\frt^*_{\bbR})$ are exactly $2\rho_n^{(j)}$, $0\leq j\leq s-1$. The $\frk$-type $E_{\mu}$ is called unitarily small (\emph{u-small} for short) if its highest weight $\mu$ lives in the unitarily small polyhedron. Many equivalent characterizations of u-small $\frk$-types are given in Theorem 6.7 of \cite{SV}.

Let $\xi_{1}, \dots, \xi_{{\rm rank} (\frg_0)}$ be the fundamental weights for a positive root system $\Delta^+(\frg, \frh)$ which restricts to $\Delta^+(\frg, \frt)$. Note that the Weyl group $W(\frg, \frt)$ can be identified as the subgroup $W_{\theta}$ of $W(\frg, \frh)$ consists of elements commuting with $\theta$. See for instance Section 2 of \cite{HKP}.  Motivated by Conjecture B of \cite{D17}, the following result generalizes Lemma 3.4 of \cite{DW}.

\begin{lemma}\label{lemma-inf-char-u-small}
Let $\Lambda=\sum_{i=1}^{{\rm rank}(\frg_0)} \lambda_i \xi_i $ be such that $0\leq \lambda_i\leq 1$ for each $1\leq i\leq {\rm rank}(\frg_0)$. Assume that $\Lambda\in\frt^*$.  Let $V_{\delta}$ be a $K$-type such that $\{\delta-\rho_n^{(j)}\}+\rho_c=w\Lambda$, for some  $0\leq j \leq s-1$ and some $w\in W(\frg, \frt)^1$.
Then $V_{\delta}$ must be u-small.
\end{lemma}

\begin{proof}
By assumption, there exists $w_1\in W(\frk,\frt)$ such that
$$
\delta=\rho_n^{(j)}+w_1(w\Lambda-\rho_c).
$$
Now for any $0\leq j^\prime \leq s-1$, we have that
\begin{align*}
\delta-2\rho_n^{(j^\prime)}
&=w_1(w\Lambda-\rho_c)+(w^{(j)}\rho-\rho_c)-2(w^{(j^\prime)}\rho-\rho_c)\\
&=w_1(w\Lambda-\rho_c)+\rho_c+w^{(j)}\rho-2w^{(j^\prime)}\rho\\
&=w\Lambda-\big( (w\Lambda -\rho_c)-w_1(w\Lambda -\rho_c)\big)+w^{(j)}\rho-2w^{(j^\prime)}\rho.
\end{align*}
Therefore,
\begin{equation}\label{delta-pairing}
\langle \delta-2\rho_n^{(j^\prime)}, w^{(j^\prime)}\xi_i\rangle
=\langle  w\Lambda+w^{(j)}\rho,  w^{(j^\prime)}\xi_i\rangle - \langle  (w\Lambda -\rho_c)-w_1(w\Lambda -\rho_c),  w^{(j^\prime)}\xi_i\rangle -2\langle \rho, \xi_i \rangle.
\end{equation}
By Lemma 5.5 of \cite{DH11}, $\xi_i-w^{-1}w^{(j^\prime)}\xi_i$ is a non-negative linear combination of roots in $\Delta^+(\frg, \frh)$.
Thus
$$
\langle \Lambda, \xi_i-w^{-1}w^{(j^\prime)}\xi_i \rangle\geq 0.
$$
Therefore,
$$
\langle w\Lambda, w^{(j^\prime)}\xi_i \rangle
=\langle \Lambda, w^{-1}w^{(j^\prime)}\xi_i \rangle
=\langle \Lambda, \xi_i \rangle - \langle \Lambda, \xi_i-w^{-1}w^{(j^\prime)}\xi_i \rangle
\leq \langle \Lambda, \xi_i \rangle  \leq \langle \rho, \xi_i \rangle.
$$
Similarly, one deduces that
$$
\langle  w^{(j)}\rho,  w^{(j^\prime)}\xi_i\rangle \leq \langle \rho, \xi_i \rangle
$$
and that
$$
\langle (w\Lambda -\rho_c)-w_1(w\Lambda -\rho_c),  w^{(j^\prime)}\xi_i\rangle \geq 0.
$$
Substituting the above  inequalities into \eqref{delta-pairing} gives that
$\langle \delta-2\rho_n^{(j^\prime)}, w^{(j^\prime)} \rangle\leq 0$. Therefore, the $K$-type $V_{\delta}$ is u-small by Theorem 6.7(e).
\end{proof}

\section{Cohomolgical induction}\label{sec-coho-ind}
This section aims to briefly recall cohomological induction, which is an effective way of constructing unitary representations.
Firstly, let us fix an element $H\in \frt_{\bbR}$, and define the $\theta$-stable parabolic subalgebra
\begin{equation}\label{def-theta-stable-parabolic}
\frq=\frl\oplus \fru
\end{equation}
as the nonnegative eigenspaces of ${\rm ad}(H)$. The Levi subalgebra $\frl$ of $\frq$ is the zero eigenspace of ${\rm ad}(H)$, while the nilradical $\fru$ of $\frq$ is the sum of positive eigenspaces of ${\rm ad}(H)$. If we denote  by $\overline{\fru}$ the sum of negative eigenspaces of ${\rm ad}(H)$, then
$$
\frg=\overline{\fru} \oplus \frl \oplus \fru.
$$
Let $L=N_{G}(\frq)$. Then $L\cap K$ is a maximal compact subgroup of $L$.

We arrange the positive root system $\Delta^+(\frg, \frh)$ so that
$$
\Delta^+(\frl, \frh)=\Delta(\frl, \frh)\cap \Delta^+(\frg, \frh), \quad\Delta(\fru, \frh)\subseteq \Delta^+(\frg, \frh).
$$
Denote by $\rho^L$ (resp., $\rho_c^L$) the half sum of positive roots in $\Delta^+(\frl, \frh)$ (resp., $\Delta^+(\frl\cap\frk, \frt)$). Let $\rho_n^L=\rho^L-\rho_c^L$. Denote bu $\rho(\fru)$ (resp., $\rho(\fru\cap\frp)$, $\rho(\fru\cap\frk)$) the half sum of roots in $\Delta(\fru, \frt)$ (resp., $\Delta(\fru\cap \frp, \frt)$, $\Delta(\fru\cap\frk, \frt)$).
The following relations hold
\begin{equation}\label{half-sum-relations}
\rho=\rho^L+\rho(\fru), \quad \rho_c=\rho_c^L+\rho(\fru\cap\frk), \quad \rho_n=\rho_n^L+\rho(\fru\cap\frk).
\end{equation}

Now let $Z$ be an admissible $(\frl, L\cap K)$-module with \emph{real} infinitesimal character $\lambda_L$. That is, $\lambda_L\in\frh_{\bbR}^*$. Assume that $\lambda_L$ is dominant for $\Delta^+(\frl, \frh)$. For simplicity, assume that $Z$ is in the \emph{good range}. That is,
\begin{equation}\label{good-range}
\langle \lambda_L +\rho(\fru), \alpha\rangle >0, \quad \forall \alpha\in \Delta(\fru, \frh).
\end{equation}
The cohomological induction functors $\caL_j(\cdot)$ and $\caR^j(\cdot)$ lift $Z$ to $(\frg, K)$-modules, and the most interesting case happens at the middle degree $S:=\dim (\fru\cap\frk)$. We refer the reader to the book \cite{KV} for detailed descriptions. As a quick glimpse, let us state the following result.

\begin{thm}\emph{(\cite[Theorems 0.50 and 0.51]{KV})}
Suppose the admissible $(\frl, L\cap K)$-module $Z$ is in the good range. Then we have
\begin{itemize}
\item[(a)] $\caL_j(Z)=\caR^{j}(Z)=0$ for $j\neq S$;
\item[(b)] $\caL_S(Z)\cong \caR^S(Z)$ as $(\frg, K)$-modules;
\item[(c)] if $Z$ is  irreducible and unitary, then $\caL_S(Z)$ is an irreducible unitary $(\frg, K)$-module with infinitesimal character $\lambda_L+\rho(\fru)$.
\end{itemize}
\end{thm}

In the special case that the inducing module $Z$ is a one-dimensional unitary character $\bbC_{\lambda}$, we will call the corresponding $(\frg, K)$-module $\caL_S(Z)$ an $A_{\frq}(\lambda)$ module. After \cite{KV}, the module $\aq(\lambda)$ is called \emph{good} (relative to $\frq$ and $\frg$) if
\begin{equation}\label{Aqlambda-good}
\langle \lambda+\rho, \alpha \rangle > 0, \quad \forall \alpha\in \Delta(\fru, \frh),
\end{equation}
and \emph{weakly good} if
\begin{equation}\label{Aqlambda-weakly-good}
\langle \lambda+\rho, \alpha \rangle \geq 0, \quad \forall \alpha\in \Delta(\fru, \frh);
\end{equation}
The module $\aq(\lambda)$ is called \emph{fair} if
\begin{equation}\label{Aqlambda-fair}
\langle \lambda+\rho(\fru), \alpha \rangle > 0, \quad \forall \alpha\in \Delta(\fru, \frh),
\end{equation}
and \emph{weakly fair} if
\begin{equation}\label{Aqlambda-weakly-fair}
\langle \lambda+\rho(\fru), \alpha \rangle \geq 0, \quad \forall \alpha\in \Delta(\fru, \frh);
\end{equation}
The range that the inducing module $Z$ lives in is crucial for the properties of $\caL_S(Z)$. We will recall some other ranges later when necessary. However, it is now worth mentioning that as shown by Salamanca-Riba \cite{Sa}, any irreducible unitary $(\frg, K)$-module with a real, integral, and strongly regular infinitesimal character $\Lambda$ must be isomorphic to an $A_{\frq}(\lambda)$ module in the good range.  Here  $\Lambda$ being \emph{strongly regular} means that
\begin{equation}\label{strongly-regular}
\langle \Lambda-\rho, \alpha \rangle \geq 0, \quad \forall \alpha\in \Delta^+(\frg, \frh).
\end{equation}

\section{Dirac index}\label{sec-DI}
In this section, we further assume that $G$ is equal rank -- that is, $\frh=\frt$ and $\fra=0$,
so that one can talk about Dirac index. Note that for general $G$ that are not necessarily of equal rank, 
a recent work of Barbasch-Pand\v zi\' c-Trapa \cite{BPT} generalized this idea to \textit{twisted Dirac index}.

Fix a choice of $\Delta^+(\frg, \frt)$, and put
$$
\Delta^+(\frp, \frt)=\Delta^+(\frg, \frt) \cap \Delta(\frp, \frt), \quad \Delta^-(\frp, \frt)=-\Delta^+(\frp, \frt).
$$
We have the corresponding isotropic decomposition
$$
\frp=\frp^+ \oplus \frp^-,
$$
where $\frp^+=\sum_{\alpha\in \Delta^+(\frp, \frt)} \frg_{\alpha}$ and $\frp^-=\sum_{\alpha\in \Delta^-(\frp, \frt)} \frg_{\alpha}$.
Then
$$
{\rm Spin}_{G}\cong \bigwedge \frp^+ \otimes \bbC_{-\rho_n}.
$$
Any weight in ${\rm Spin}_{G}$ has the form $-\rho_n + \langle\Phi\rangle$, where $\Phi$ is a subset of $\Delta^+(\frp, \frt)$ and $\langle\Phi \rangle$ stands for the sum of the roots in $\Phi$.
Now put
\begin{equation}\label{spin-module-even-odd}
{\rm Spin}_{G}^+ = \bigwedge^{\rm even} \frp^+ \otimes \bbC_{-\rho_n}, \quad {\rm Spin}_{G}^- = \bigwedge^{\rm odd} \frp^+ \otimes \bbC_{-\rho_n}.
\end{equation}
Let $X$ be any $(\frg, K)$-module, the Dirac operator $D$ interchanges $X\otimes {\rm Spin}^+_{G}$ and $X\otimes {\rm Spin}^-_{G}$. Thus the Dirac cohomology $H_D(X)$ breaks up into the even part and the odd part, which will be denoted by $H_D^+(X)$ and $H_D^-(X)$ respectively. The Dirac index is defined as
\begin{equation}\label{Dirac-index}
{\rm DI}(X):=H_D^+(X)-H_D^-(X),
\end{equation}
which is a virtual $\widetilde{K}$-module.
By Remark 3.8 of \cite{MPVZ}, if another positive root system $(\Delta^+)^\prime(\frg, \frt)$ is chosen,
$$
{\rm DI}^\prime(X)=(-1)^{\# ((\Delta^+)^\prime(\frp, \frt) \setminus \Delta^+(\frp, \frt)) }{\rm DI}(X).
$$
Therefore, the Dirac index is well-defined up to a sign. Moreover, by Proposition 3.12 of \cite{MPV},
$$
{\rm DI}(X)=X\otimes {\rm Spin}^+_{G}-X\otimes {\rm Spin}^-_{G}.
$$

It is obvious that if ${\rm DI}(X)$ is non-zero, then $H_D(X)$ must be non-zero. However, the converse is not true. Indeed, for any irreducible $(\frg, K)$-module $X$, Conjecture 10.3 of \cite{H15} asserts that there should be no cancellation between $H_D^+(X)$ and $H_D^-(X)$. A counter-example to this conjecture (revisited below in Example \ref{eg-splitf4}) has been reported in \cite{DDY} on split $F_4$, where the Dirac cohomology is non-zero, but cancellation happens and the Dirac index vanishes.

The following result suggests that for most cases, we can still expect that there is no cancellation when passing from Dirac cohomology to Dirac index.

\begin{prop}
Let $Z$ be an irreducible unitary $(\frl, L\cap K)$-module with infinitesimal character $\lambda_L\in \frt_{\bbR}^*$ which is weakly good. That is, $\langle \lambda_L + \rho(\fru), \alpha \rangle\geq 0$ for all $\alpha\in \Delta(\fru, \frt)$. Assume that there is no cancellation between $H_D^+(Z)$ and $H_D^-(Z)$. Then there is no cancellation between $H_D^+(\caL_S(Z))$ and $H_D^-(\caL_S(Z))$ as well.
\end{prop}
\begin{proof}
Note firstly that $L$ is equal rank as well. We have that
$$
{\rm Spin}_{L}\cong \wedge^\bullet (\frl \cap \frp^+)\otimes \bbC_{-\rho_n^{L}}, \quad \rho_n=\rho_n^{L}+\rho(\fru\cap\frp).
$$
Now let us go through the proof of \cite[Lemma 4.3]{DH15}. Without loss of generality, we assume that $H_D(Z)$ and $H_D(\caL_S(Z))$ are both non-zero.

On the $L$ level, take the highest weight of an arbitrary component in $H_D(Z)$
\begin{equation}\label{gamma-L}
\gamma_L:=\mu_L -\rho_n^{L}+\langle\Phi\rangle,
\end{equation}
where $\mu_L$ is the highest weight of a $L\cap K$-type in $Z$, and $\Phi$ is a subset of  $\Delta^+(\frl\cap\frp, \frt)$ (see Remark 4.4 of \cite{DH15}). Therefore, the component $\gamma_L$ has sign $(-1)^{\#\Phi}$ in ${\rm DI} (Z)$.

If $\mu_L+2\rho(\fru\cap\frp)$ is dominant for $\Delta^+(\frk, \frt_f)$, then on the $G$ level, we have the  following counterpart of $\gamma_L$:
\begin{equation}\label{gamma-G}
\gamma_G:= \big(\mu_L +2\rho(\fru\cap\frp)\big)-\rho_n+\langle\Phi\rangle,
\end{equation}
where $\Phi$ is the \emph{same} subset of $\Delta^+(\frl\cap\frp, \frt)$  as in \eqref{gamma-L}. In particular, the component $\gamma_G$ has sign $(-1)^{\#\Phi}$ in ${\rm DI}(\caL_S(Z))$ as well.
Note that $\gamma_G=\gamma_L+\rho(\fru\cap\frp)$ in view of \eqref{half-sum-relations}. On the other hand, if $\mu_L+2\rho(\fru\cap\frp)$ is not dominant for $\Delta^+(\frk, \frt)$, then $\gamma_L$ has no counterpart on the $G$ level.

To sum up, after moving from $L$ to $G$, any component in $H_D^+(Z)$ and $H_D^-(Z)$ is either killed, or just lifted by $\rho(\fru\cap\frp)$ with the parity being preserved.
Thus the desired result follows.
\end{proof}

\subsection{Translation Principle of Dirac Index}
The main goal of this manuscript is to understand the Dirac index of $\aq(\lambda)$-modules beyond admissible range.
To do so, one can first compute the Dirac index of $\aq(\lambda)$-modules by \cite[Theorem 5.2]{HKP}, and then apply the translation principle of Dirac index given in \cite[Theorem 4.7]{MPV}. Note that the translation principle works well only for Dirac index, but not for
Dirac cohomology.

\smallskip

The Dirac index for $\aq(\lambda)$-modules in the admissible range is given as follows:
\begin{prop} \label{prop-index}
Let $\lambda\in \frt^*$ be admissible. That is, $\langle \lambda, \alpha \rangle\geq 0$ for all $\alpha\in \Delta(\fru, \frt)$. Then
\begin{equation}
{\rm DI}(A_{\frq}(\lambda))=(-1)^{\#\Delta^+(\frl\cap\frp, \frt)}\sum_{w\in W(\frl, \frt)^1} {\rm det}(w) E_{w\rho-\rho_c}.
\end{equation}
\end{prop}
\begin{proof}
This result is actually implicit in the proof of \cite[Theorem 5.2]{HKP}.
Adopting the notations there, $E_{\lambda + 2\rho(\fru\cap\frp)}$ is the lowest $K$-type of $A_{\frq}(\lambda)$; moreover, for each $w\in W(\frl, \frt)^1$,  the weight $\lambda+ 2\rho(\fru\cap\frp)+\nu$, where $\nu:=-\rho_n+\langle\Phi\rangle$,
shows up as the highest weight of the $\widetilde{K}$-type $E_{w\rho-\rho_c}$ of $H_D(A_{\frq}(\lambda))$. The sign of $E_{w\rho-\rho_c}$ in the Dirac index of $A_{\frq}(\lambda)$ is $(-1)^{\#\Phi}$. As shown on page 169 of \cite{HKP},
$$
\Phi=\Delta^+(\frl\cap\frp, \frt)\setminus \Phi_w,
$$
where $\Phi_w=w(\Delta^-(\frg, \frt))\cap \Delta^+(\frg, \frt)$. Thus $(-1)^{\#\Phi}=(-1)^{\# \Delta^+(\frl\cap\frp, \frt)}(-1)^{\#\Phi_w}$. Now the result follows since $(-1)^{\# \Phi_w}=\det(w)$.
\end{proof}

To go beyond the admissible range, let
\begin{equation}\label{E-tilde-mu}
\widetilde{E}_{\mu} = \begin{cases}
0 & \text{if}\ \mu\ \text{is}\ \Delta(\mathfrak{k},\mathfrak{t})\text{-singular} \\
\det(w) E_{w\mu - \rho_c} & \text{if } \exists w \in W(\mathfrak{k},\mathfrak{t})\ \text{s.t. } w\mu\ \text{is dominant regular for } \Delta^+(\mathfrak{k},\mathfrak{t})
\end{cases},
\end{equation}
so that the Dirac index of $\aq(0)$ in Proposition \ref{prop-index} can be re-written as
\begin{equation} \label{Aq0DI}
{\rm DI}(A_{\frq}(0))= \sum_{w\in W(\frl, \frt)^1} \det(w) \widetilde{E}_{w\rho}
\end{equation}
up to a sign. Then we have:
\begin{thm}\label{thm-DI-Aqlambda}
The Dirac index of weakly fair $\aq(\lambda)$ is equal to
$$
{\rm DI}(A_{\frq}(\lambda)) = \sum_{w \in W(\mathfrak{l},\mathfrak{t})^1} {\rm det}(w) \widetilde{E}_{w(\lambda + \rho)}.$$
\end{thm}
\begin{proof}
If $\aq(\lambda)$ is in the weakly fair range, then by Theorem 0.53 of \cite{KV},
$$
\aq(\lambda) = \mathcal{R}_{\mathfrak{q}}^S(\mathbb{C}_{\lambda}) = \sum_{i = 0}^S (-1)^{i} \mathcal{R}_{\mathfrak{q}}^{S-i}(\mathbb{C}_{\lambda}).
$$
Thus $\aq(\lambda)$ is in the same coherent family of virtual $(\mathfrak{g},K)$-modules as $\aq(0)$ (c.f. \cite[Definition 3.1]{MPVZ}). The result follows from Theorem 4.7 of \cite{MPV}.
\end{proof}

\begin{example} \label{eg-splitf4}
Consider the equal rank connected group $\texttt{F4\_s}$ in $\texttt{atlas}$, and let us revisit Example 6.3 of \cite{DDY} via the  translation method.  Fix  $\Delta^+(\frg, \frt)$ from page 716 of \cite{Kn} so that it has the following simple roots:
$$
\alpha_1=\frac{1}{2}(e_1-e_2-e_3-e_4), \quad \alpha_2=e_4, \quad \alpha_3=e_3-e_4, \quad \alpha_4=e_2-e_3.
$$
Fix $\Delta^+(\frk, \frt)$ so that it has the following simple roots:
$$
\gamma_1=\alpha_1, \quad\gamma_2=\alpha_2, \quad \gamma_3=\alpha_3, \quad
\gamma_4=2 \alpha_1   + 4 \alpha_2  + 3 \alpha_3     + 2 \alpha_4=e_1+e_2.
$$
Then
$$
\rho_c=(2, -1, \frac{3}{2}, \frac{1}{2}).
$$

Let us denote the simple reflections $s_{\alpha_i}$ by $s_i$ for $1\leq i\leq 4$.
The set $W(\frg, \frt)^1$ can be enumerated  in the way of \eqref{Wgtone-label} as follows:
\begin{align*}
W(\frg, \frt)^1=\{e, s_4, s_4s_3, s_4s_3s_2, s_4s_3s_2s_1, s_4s_3s_2s_3, s_4s_3s_2s_1s_3, s_4s_3s_2s_3s_4, \\ s_4s_3s_2s_1s_3s_2, s_4s_3s_2s_1s_3s_4, s_4s_3s_2s_1s_3s_2s_3, s_4s_3s_2s_1s_3s_2s_4\}.
\end{align*}
Correspondingly, we have that
\begin{align*}
\rho^{(0)}=(\frac{11}{2}, \frac{5}{2}, \frac{3}{2}, \frac{1}{2}), \quad \rho^{(1)}=(\frac{11}{2}, \frac{3}{2}, \frac{5}{2}, \frac{1}{2}), \quad \rho^{(2)}=(\frac{11}{2}, \frac{1}{2}, \frac{5}{2}, \frac{3}{2}), \quad
\rho^{(3)}=(\frac{11}{2}, -\frac{1}{2}, \frac{5}{2}, \frac{3}{2}),\\
\rho^{(4)}=(5, -1, 3, 2), \quad
\rho^{(5)}=(\frac{11}{2}, -\frac{3}{2}, \frac{5}{2}, \frac{1}{2}), \quad
\rho^{(6)}=(5, -2, 3, 1), \quad
\rho^{(7)}=(\frac{11}{2}, -\frac{5}{2}, \frac{3}{2}, \frac{1}{2}),\\
\rho^{(8)}=(\frac{9}{2}, -\frac{5}{2}, \frac{7}{2}, \frac{1}{2}), \quad
\rho^{(9)}=(5, -3, 2, 1), \quad
\rho^{(10)}=(\frac{7}{2}, -\frac{5}{2}, \frac{9}{2}, \frac{1}{2}), \quad
\rho^{(11)}=(\frac{9}{2}, -\frac{7}{2}, \frac{5}{2}, \frac{1}{2}).
\end{align*}

The element $H = (1,0,1,0) \in i\frt_0$ defines a $\theta$-stable parabolic subalgebra
$$
\mathfrak{q} = \mathfrak{l} + \mathfrak{u}
$$
via  the way of \eqref{def-theta-stable-parabolic}, where $[\frl, \frl]=\frs\frp(6, \bbR)$. Using \eqref{Aq0DI}, one computes that
$$
{\rm DI}(\aq(0))=
-\widetilde{E}_{\rho^{(1)}}+\widetilde{E}_{\rho^{(2)}}-\widetilde{E}_{\rho^{(3)}}
+\widetilde{E}_{\rho^{(4)}}+\widetilde{E}_{\rho^{(5)}}-\widetilde{E}_{\rho^{(6)}}
+\widetilde{E}_{\rho^{(8)}}-\widetilde{E}_{\rho^{(10)}}.
$$
Now let us consider the module $\aq(\lambda)$ with
$\lambda = (-3,0,-3,0)$. This module is in the fair range.
Theorem \ref{thm-DI-Aqlambda} says that
$$
{\rm DI}(\aq(\lambda))=
-\widetilde{E}_{\rho^{(1)}+\lambda}+\widetilde{E}_{\rho^{(2)}+\lambda}
-\widetilde{E}_{\rho^{(3)}+\lambda}+\widetilde{E}_{\rho^{(4)}+\lambda}+
\widetilde{E}_{\rho^{(5)}+\lambda}-\widetilde{E}_{\rho^{(6)}+\lambda}
+\widetilde{E}_{\rho^{(8)}+\lambda}-\widetilde{E}_{\rho^{(10)}+\lambda}.
$$
One can check that the $\rho^{(j)} + \lambda$ above is $\Delta^+(\frk, \frt)$-singular except for $j = 3$ and $10$. Hence
$$
\mathrm{DI}(\aq(\lambda)) = -\widetilde{E}_{\rho^{(3)} + \lambda} - \widetilde{E}_{\rho^{(10)} + \lambda}
= -\widetilde{E}_{(\frac{5}{2}, -\frac{1}{2}, -\frac{1}{2}, \frac{3}{2})} - \widetilde{E}_{(\frac{1}{2}, -\frac{5}{2}, \frac{3}{2}, \frac{1}{2})}.
$$
Note that
$$
s_{\gamma_2}s_{\gamma_3}
(\frac{5}{2}, -\frac{1}{2}, -\frac{1}{2}, \frac{3}{2})=s_{\gamma_4}(\frac{1}{2}, -\frac{5}{2}, \frac{3}{2}, \frac{1}{2})=(\frac{5}{2}, -\frac{1}{2}, \frac{3}{2}, \frac{1}{2}),
$$
which is dominant regular for $\Delta^+(\frk, \frt)$. Therefore, by \eqref{E-tilde-mu},
$$
\mathrm{DI}(\aq(\lambda)) = -E_{(\frac{5}{2}, -\frac{1}{2}, \frac{3}{2}, \frac{1}{2})-\rho_c} +E_{(\frac{5}{2}, -\frac{1}{2}, \frac{3}{2}, \frac{1}{2})-\rho_c}=0.
$$

Finally, we note that $\aq(0)$ is the module in the $15$-th row of \cite[Table 10]{DDY} with $a=1$, while $\aq(\lambda)$ is the module in the $13$-th row of \cite[Table 6]{DDY}.
\hfill\qed
\end{example}

\section{Dirac index of weakly fair $A_{\frq}(\lambda)$ modules of $U(p, q)$}
From now on, we focus on $G = U(p,q)$.
Consider the $\theta$-stable parabolic subalgebra $\mathfrak{q} = \mathfrak{l} \oplus \mathfrak{u}$ of $\frg$ defined in the way of \eqref{def-theta-stable-parabolic} by the following element
\begin{equation} \label{thetalevi}
(\overbrace{a_1, \dots, a_1}^{p_1}, \overbrace{a_2, \dots, a_2}^{p_2}, \dots, \overbrace{a_k \dots a_k}^{p_k} | \overbrace{a_1 \dots a_1}^{q_1}, \overbrace{a_2, \dots, a_2}^{q_2}, \dots, \overbrace{a_k \dots a_k}^{q_k}),
\end{equation}
where $p_i, q_j \geq 0$ for all $1 \leq i, j \leq k$ and $a_1 > a_2 > \dots > a_k$. Here $\sum_{i=1}^k p_i=p$, and $\sum_{i=1}^k q_i=q$. Then
$$
\mathfrak{l} = u(p_1,q_1) \oplus u(p_2,q_2) \oplus \dots \oplus u(p_k,q_k).
$$

\subsection{A Characterization of $\aq(\lambda)$-modules}
\begin{definition} \label{def-chains}
Let $\mathfrak{q} = \mathfrak{l} \oplus \mathfrak{u}$ be $\theta$-stable Levi subalgebra of $\mathfrak{g}$ given by Equation \eqref{thetalevi}. Then a {\bf signed chain} is given by:
$$\mathcal{C}_i(A_i) := (A_i, A_i-1,\dots, a_i + 1, a_i)^{p_i,q_i}, \quad i = 1,\dots, k,$$
where $a_i = A_i - (p_i + q_i - 1)$. An {\bf $\aq$-chain} is defined as the \emph{ordered}
union of signed chains $\bigcup_{i=1}^{k} \mathcal{C}_i(A_i)$.
Moreover, we say an $\aq$-chain is
\begin{itemize}
\item[(a)] {\bf good} if $a_i > A_{i+1}$ for all $i$.
\item[(b)] {\bf fair} if $(A_i + a_i)/2 > (A_{i+1} + a_{i+1})/2$ for all $i$.
\end{itemize}
If the above strict inequalities are replaced by $\geq$, we call the
corresponding $\aq$-chain {\bf weakly good} and {\bf weakly fair} respectively.
\end{definition}
Let $\lambda$ be such that $\mathbb{C}_{\lambda}$
is a unitary $(\mathfrak{l}, L \cap K)$-module. We assign each $\aq(\lambda)$-module to an $\aq$-chain. Firstly, when $\lambda = 0$,
$\aq(0)$ is assigned to the following $\aq$-chain:
\begin{equation}\label{chain-Aq0}
\aq(0) \longleftrightarrow \bigcup_{i=1}^k \mathcal{C}_i(Z_i)
\end{equation}
with $Z_i := \sum_{j = i}^{k} p_j + q_j$.
In general, for each $A_{\frq}(\lambda)$ with
$$\lambda := (\overbrace{\lambda_1, \dots, \lambda_1}^{p_1}, \overbrace{\lambda_2, \dots, \lambda_2}^{p_2}, \dots, \overbrace{\lambda_k \dots \lambda_k}^{p_k} | \overbrace{\lambda_1 \dots \lambda_1}^{q_1}, \overbrace{\lambda_2, \dots, \lambda_2}^{q_2}, \dots, \overbrace{\lambda_k \dots \lambda_k}^{q_k}), \quad \lambda_i \in \mathbb{Z}$$
we assign the $\aq$-chain
\begin{equation}\label{chain-Aqlambda}
\aq(\lambda) \longleftrightarrow \bigcup_{i=1}^k \mathcal{C}_i(Z_i + \lambda_i).
\end{equation}
Similar to the characterizations of \cite[p.22]{T1}, one can check that an $\aq(\lambda)$-module is in the (weakly) good/(weakly) fair range if and only if its corresponding $\aq$-chain is in the (weakly) good/(weakly) fair range in the sense of
Definition \ref{def-chains} respectively. Combining Theorem 3.1b(iv), Definition 3.4 and Lemma 3.5 of \cite{T1}, we know that a weakly fair $\aq(\lambda)$ of $U(p, q)$ is either
irreducible unitary or zero.

Let
$$\tau := (\frac{p+q+1}{2}, \dots, \frac{p+q+1}{2}).$$
Then the infinitesimal character of $\aq(\lambda) \longleftrightarrow \bigcup_{i=1}^k \mathcal{C}_i(A_i)$ is given by the
$W(\frg, \frt)$-conjugate of
$$(A_1, \dots, a_1; \dots; A_k, \dots, a_k) - \tau.$$
For instance, the infinitesimal character of $\aq(0)$ is equal to
$$(Z_1,\dots,Z_1 - (p_1 + q_1 -1); \cdots; Z_k,\dots,Z_k - (p_k + q_k -1))- \tau = (p+q , p+q-1, \dots, 2,1) - \tau = \rho.$$
Using this characterization of the infinitesimal character of $\aq(\lambda)$ modules, the results
of \cite{HP} can be rephrased as follows.
\begin{lemma}\label{lemma-Aqlambda-nonzeroDirac}
Suppose $A_{\frq}(\lambda)$ corresponds to the $\aq$-chain $\bigcup_{i=1}^k \mathcal{C}_i(A_i)$.
Assume that $H_D(\aq(\lambda))$ $\neq$ $0$. Then
\begin{equation} \label{hpcondition}
\begin{aligned}
&\bullet\ \text{The entries}\ \bigcup_{i=1}^k \{A_i,\dots, a_i\}\ \text{cannot appear more than twice; and}\\
&\bullet\ \text{There are at most}\ \min\{p, q\}\ \text{distinct entries appearing twice in}\ \bigcup_{i=1}^k \{A_i,\dots, a_i\}.
\end{aligned}
\end{equation}
\end{lemma}
\begin{proof}
Let $\Lambda := (A_1, \dots, a_1; \dots; A_k, \dots, a_k) - \tau$ be the infinitesimal character of $\aq(\lambda)$. If $H_D(\aq(\lambda)) \neq 0$, then by Theorem \ref{thm-HP} there must be an element
$w \in W(\mathfrak{g},\mathfrak{t})$ such that $w\Lambda$ is a $\Delta^+(\mathfrak{k},\mathfrak{t})$-regular weight
of the form:
$$w\Lambda = (\omega_1, \dots, \omega_p |\ \omega_{p+1},\dots, \omega_{p+q}).$$
Suppose on the contrary that there are $r\geq 3$ repeated entries in
$\bigcup_{i=1}^k \{A_i,\dots, a_i\}$, then {\it all} Weyl conjugates $w\Lambda$
must have repeated entries in either $\{\omega_1, \dots, \omega_p\}$
or $\{\omega_{p+1},\dots,\omega_{p+q}\}$. Hence $w\Lambda$ cannot be $\Delta^+(\mathfrak{k},\mathfrak{t})$-regular.

Similarly, if there are more than $\min\{p,q\}$ entries of $\bigcup_{i=1}^k \{A_i,\dots, a_i\}$ repeating twice, then
at least one of them will appear twice in either $\{\omega_1, \dots, \omega_p\}$
or $\{\omega_{p+1},\dots,\omega_{p+q}\}$, contradicting the regularity of $w\Lambda$.
\end{proof}

\begin{example}\label{exam-su24}
Let $G = U(2,4)$ and consider the $\theta$-stable parabolic subalgebra $\frq$ of $\frg$ defined by the element $(1, 1 |2, 1, 1, 0)$. Then $\frl=\fru(0,1) + \fru(2,2) + \fru(0, 1)$. We assign the chain
$$(6)^{0,1} \cup (5,4,3,2)^{2,2} \cup (1)^{0,1}$$
to the module $A_{\frq}(0)$, and assign the chain
$$(4)^{0,1} \cup (5,4,3,2)^{2,2} \cup (3)^{0,1}$$
to the module $A_{\frq}(\lambda)$ with $\lambda=(0,0|-2, 0, 0, 2)$. Note that $A_{\frq}(\lambda)$ is fair but not weakly good.
The infinitesimal character of $\aq(\lambda)$ is $(\frac{3}{2},\frac{1}{2},\frac{1}{2},-\frac{1}{2},-\frac{1}{2},-\frac{3}{2})$.\hfill\qed
\end{example}

\subsection{Dirac index for weakly fair $\aq(\lambda)$ modules}
We now rephrase Theorem \ref{thm-DI-Aqlambda} using our notion of chains.

A $(p,q)$-\emph{shuffle} of the numbers $a_1 > a_2 > \cdots >a_{p+q}$ is a permutation $i_1, \dots, i_p$, $j_1, \dots, j_q$  of the indices $1, 2, \dots, p+q$ such that
$$
a_{i_1}> a_{i_2}>\cdots>a_{i_p}, \quad a_{j_1}>a_{j_2}>\cdots>a_{j_q}.
$$
The total number of $(p,q)$-shuffles is ${p+q\choose p}$.
\begin{cor} \label{cor-diupq}
For all weakly fair $\aq(\lambda)$ corresponding to $\aq$-chain $\bigcup_{i=1}^k \mathcal{C}(A_i)$ in \eqref{chain-Aqlambda},
its Dirac index is given by:
$$
\mathrm{DI}(\aq(\lambda)) = \sum_{w \in W(\mathfrak{l},\mathfrak{t})^1} {\rm det}(w) \widetilde{E}_{
(\underbrace{\alpha_1, \dots, \beta_1}_{p_1};\ \dots;\ \underbrace{\alpha_k, \dots, \beta_k}_{p_k} |
\underbrace{\gamma_1, \dots, \delta_1}_{q_1};\ \dots;\ \underbrace{\gamma_k, \dots, \delta_k}_{q_k}) - \tau}$$
where the sum above is over all
are the $(p_i, q_i)$-shuffles of the $p_i + q_i$ coordinates $(\alpha_i, \dots, \beta_i; \gamma_i, \dots, \delta_i)$ of $(A_i, \dots, a_i)$ given by $w \in W(\mathfrak{l},\mathfrak{t})^1$.
\end{cor}

\begin{example}
Let $G$ be $U(2,4)$, and consider the $A_{\frq}(\lambda)$ module
corresponding to $(4)^{0,1} \cup (5,4,3,2)^{2,2} \cup (3)^{0,1}$
in Example \ref{exam-su24}. Then $\lambda = (0,0;-2,0,0,2)$, and its Dirac index is equal to
\begin{align*}
{\rm DI}(\aq(\lambda)) =\ & \widetilde{E}_{(54|6;32;1) +\lambda - \tau} - \widetilde{E}_{(53|6;42;1)+\lambda - \tau} + \widetilde{E}_{(52|6;43;1)+\lambda - \tau} \\
& +  \widetilde{E}_{(43|6;52;1)+\lambda - \tau} - \widetilde{E}_{(42|6;53;1)+\lambda - \tau} + \widetilde{E}_{(32|6;54;1)+\lambda - \tau}\\
=\ & \widetilde{E}_{(54|4;32;3)- \tau} - \widetilde{E}_{(53|4;42;3)- \tau} + \widetilde{E}_{(52|4;43;3)- \tau} \\
 & + \widetilde{E}_{(43|4;52;3)- \tau} - \widetilde{E}_{(42|4;53;3)- \tau} + \widetilde{E}_{(32|4;54;3)- \tau}.
\end{align*}
Note that only the fourth term has non-singular weight for $\Delta(\frk, \frt)$. Thus we have
$${\rm DI}(\aq(\lambda)) = \widetilde{E}_{(43|4;52;3)- \tau} = E_{(43|5432)- \tau - \rho_c} =  E_{(00|0000)}$$
up to a sign, where the second equality holds since the permutation $(4523) \mapsto (5432)$ is even. \hfill\qed
\end{example}

We now compute the Dirac index for all weakly fair $\aq(\lambda)$, i.e., the corresponding
$\aq$-chains $\bigcup_{i=1}^k \mathcal{C}_i(A_i)$ must be weakly fair in the sense of Definition \ref{def-chains}.
First of all, if the $\aq$-chain does not satisfy \eqref{hpcondition}, then by Lemma \ref{lemma-Aqlambda-nonzeroDirac}
it must have zero Dirac cohomology and Dirac index. Therefore, we assume our weakly fair $\aq$-chain
satisfies \eqref{hpcondition} from now on.

Let $\mathcal{R}$ be the coordinates appearing more than once in $\bigcup_i \{A_i,\dots,a_i\}$.
By Lemma \ref{lemma-Aqlambda-nonzeroDirac}, $\#\mathcal{R} \leq \min\{p, q\}$, and each $r \in \mathcal{R}$ can only appear in
two chains. For all $i < j$, let
\begin{equation} \label{rij}
\mathcal{R}_{ij} := \{A_i,\dots,a_i\} \cap \{A_j,\dots,a_j\}.
\end{equation}
Then the disjoint union of the non-empty $\mathcal{R}_{ij}$'s gives $\mathcal{R}$.

\begin{lemma}\label{lemma-30}
Using the above notations, consider the inequalities for
non-negative integers $(a_{ij}, b_{ij})$, $1 \leq i < j \leq k$:
\begin{equation} \label{eq-inequality}
\begin{cases}
a_{ij} + b_{ij} = \#\mathcal{R}_{ij}, \\
a_{ij} \leq \min\{p_i,q_j\},\\
b_{ij} \leq \min\{p_j,q_i\}, \\
\left(\sum_{j > i} a_{ij} + \sum_{l < i} b_{li}\right) \leq p_i, \\
\left(\sum_{l < i} a_{li} + \sum_{j > i} b_{ij}\right) \leq q_i.
\end{cases}
\end{equation}
Suppose the above inequalities have no solutions, then $\mathrm{DI}(\aq(\lambda)) = 0$.
\end{lemma}
\begin{proof}
By the translation principle, in order for $\mathrm{DI}(\aq(\lambda)) \neq 0$,
there must be some $w \in W(\mathfrak{l},\mathfrak{t})^1$ such that
$\widetilde{E}_{w\rho} \longrightarrow \widetilde{E}_{w(\rho + \lambda)}$ yields a
non-singular weight in $\Delta^+(\frk, \frt)$. Therefore, the entries in
$\mathcal{R}$ must appear both in the first $p$ coordinates and the
last $q$ coordinates of $w(\rho + \lambda)$.

Indeed, each solution set $(a_{ij}, b_{ij})$ in \eqref{eq-inequality}
determines the positions of the entries of $\mathcal{R}$ in $w(\rho + \lambda)$.
More precisely, $a_{ij}$ gives the possibilities of the cardinality of a subset
$\mathcal{S}_{ij} \subset \mathcal{R}_{ij}$, so that the entries in $\mathcal{S}_{ij}$ appear
between the $(p_1 + \dots + p_{i-1} + 1)^{st}$ and $(p_1 + \dots + p_{i})^{th}$ coordinates
as well as the $(p + q_1 + \dots + q_{j-1} + 1)^{st}$ and $(p + q_1 + \dots + q_{j})^{th}$
coordinates of $w(\rho + \lambda)$, while $b_{ij}$ is the cardinality of
$\mathcal{R}_{ij} \backslash \mathcal{S}_{ij}$, whose entries appear between the
$(p_1 + \dots + p_{j-1} + 1)^{st}$ and $(p_1 + \dots + p_{j})^{th}$ coordinates
as well as the $(p + q_1 + \dots + q_{i-1} + 1)^{st}$ and $(p + q_1 + \dots + q_{i})^{th}$
coordinates of $w(\rho + \lambda)$ (See Algorithm \ref{indexalg} below for more details).
Consequently, if \eqref{eq-inequality} have no solutions, then for all $w \in W(\mathfrak{l},\mathfrak{t})^1$,
there exists $r_w \in \mathcal{R}$ such that it appears in
the first $p$ coordinates or the last $q$ coordinates of $w(\rho + \lambda)$ twice,
and hence the Dirac index is zero.
\end{proof}

\begin{example}
Let $G=U(2,2)$, and we consider the $6$ discrete series representations $A_{\mathfrak{b}_i}(0)$ for $\mathfrak{b}_i = \mathfrak{h}_i + \mathfrak{n}_i$
corresponding to the $\theta$-stable Borel subalgebras defined by the elements $\rho_1 = (4,3|2,1)$, $\rho_2 = (4,2|3,1)$, $\rho_3 =(4,1|3,2)$, $\rho_4 = (3,2|4,1)$, $\rho_5 = (3,1|4,2)$ and $\rho_6 = (2,1|4,3)$ in the fundamental Cartan subalgebra.
In terms of $\aq$-chains, they are of the form
$$A_{\mathfrak{b}_i}(0) = (4)^{\sigma_{i,4}} \cup (3)^{\sigma_{i,3}} \cup (2)^{\sigma_{i,2}} \cup (1)^{\sigma_{i,1}},$$
where
$$\sigma_{i,j} = \begin{cases}
(1,0) & \text{if}\ j\ \text{appears in the first 2 coordinates of}\ \rho_i\\
(0,1) & \text{if}\ j\ \text{appears in the last 2 coordinates of}\ \rho_i
\end{cases},$$
and $H_D(A_{\mathfrak{b}_i}(0)) = \widetilde{E}_{\rho_i - \tau}$.

Suppose we apply translation
$$A_{\mathfrak{b}_i}(0) = (4)^{\sigma_{i,4}} \cup (3)^{\sigma_{i,3}} \cup (2)^{\sigma_{i,2}} \cup (1)^{\sigma_{i,1}} \longrightarrow
A_{\mathfrak{b}_i}(\lambda) = (3)^{\sigma_{i,4}} \cup (3)^{\sigma_{i,3}} \cup (2)^{\sigma_{i,2}} \cup (2)^{\sigma_{i,1}},$$
then the infinitesimal character of $A_{\mathfrak{b}_i}(\lambda)$ satisfies \eqref{hpcondition}, yet $\mathrm{DI}(A_{\mathfrak{b}_i}(\lambda)) = E_{(3,2|3,2) - \rho_c - \tau}$ is nonzero only for $i = 2,3,4,5$.

Under the perspective of the above Lemma, $\mathcal{R} = \mathcal{R}_{12} \cup \mathcal{R}_{34} = \{3\} \cup \{2\}$. Equation \eqref{eq-inequality} is solvable for the $\aq$-chains corresponding to $A_{\mathfrak{b}_2}(\lambda)$ (with $(a_{12},b_{12},a_{34},b_{34}) = (1,0,1,0)$), $A_{\mathfrak{b}_3}(\lambda)$ (with $(a_{12},b_{12},a_{34},b_{34}) = (1,0,0,1)$), and $A_{\mathfrak{b}_4}(\lambda)$ (with $(a_{12},b_{12},a_{34},b_{34}) = (0,1,1,0)$) and $A_{\mathfrak{b}_5}(\lambda)$ (with $(a_{12},b_{12},a_{34},b_{34}) = (0,1,0,1)$) only.

To see why $A_{\mathfrak{b}_1}(\lambda)$ and $A_{\mathfrak{b}_6}(\lambda)$ yield zero Dirac index, one can
use \cite{T1} to compute their associated varieties and annihilators.
In our example, the tableau corresponding to the annihilator of both modules is equal to
$\begin{tabular}{|c|c|}
\hline
3 & 2\tabularnewline
\hline
3 & 2\tabularnewline
\hline
\end{tabular}\ $,
which has repeated entries on the columns. Therefore, it is equivalent to the zero tablueau, and hence both
$A_{\mathfrak{b}_1}(\lambda)$ and $A_{\mathfrak{b}_6}(\lambda)$ are zero modules by \cite[Theorem 7.9]{T1}. \hfill\qed
\end{example}

\begin{remark}
We expect that the above observation holds in general. That is, if a weakly fair $\aq(\lambda)$ module whose infinitesimal character
satisfies \eqref{hpcondition} but \eqref{eq-inequality} has no solutions, then it must be the zero module.
\end{remark}

By Lemma \ref{lemma-30}, we can focus on the weakly fair $\aq(\lambda)$ such that
\eqref{eq-inequality} has a solution. In such cases, we list
all possible $\widetilde{K}$-types appearing in ${\rm DI}(\aq(\lambda))$:

\begin{alg} \label{indexalg}
Each term appearing in the Dirac index of $\aq(\lambda)$ is obtained by filling up the coordinates $\alpha_i$, $\beta_i$ of the form
$$(\overbrace{\alpha_1, \dots, \alpha_1}^{p_1}; \overbrace{\alpha_2, \dots, \alpha_2}^{p_2}; \dots; \overbrace{\alpha_k, \dots ,\alpha_k}^{p_k} | \overbrace{\beta_1 \dots \beta_1}^{q_1}; \overbrace{\beta_2, \dots, \beta_2}^{q_2}; \dots; \overbrace{\beta_k, \dots, \beta_k}^{q_k})$$
by the following steps:
\begin{enumerate}
\item For each solution $\{(a_{ij}, b_{ij})|\ 1\leq i < j \leq k\}$ given in \eqref{eq-inequality}, fix a subset $\mathcal{S}_{ij} \subseteq \mathcal{R}_{ij}$ with $\#\mathcal{S}_{ij}= a_{ij}$ (and hence $\#(\mathcal{R}_{ij} \backslash \mathcal{S}_{ij}) = b_{ij}$).
\item Fill in $a_{ij}$ of the positions marked with $\alpha_i$ and $\beta_j$ by $\mathcal{S}_{ij}$, and $b_{ij}$ of the positions marked with $\alpha_j$ and $\beta_i$ by $\mathcal{R}_{ij} \backslash \mathcal{S}_{ij}$.
\item For the remaining unfilled $\alpha_i$ and $\beta_i$ positions, fill them up with the coordinates $\{A_i, \dots, a_i\} \backslash \mathcal{R}$.
\item For the positions marked by $\alpha_i, \dots, \alpha_i$ ($1 \leq i \leq p$) and $\beta_j, \dots, \beta_j$ ($1 \leq j \leq q$), rearrange its entries in descending order and get $\nu$, so that $\widetilde{E}_{\nu}$ is obtained from translating an $\widetilde{E}_{w\rho}$
appearing in ${\rm DI}(\aq(0))$.
\item Reorder the first $p$ and last $q$ coordinates of $\nu$ in descending order to get a dominant regular $\Delta(\mathfrak{k},\mathfrak{t})$-weight $\mu$. Then the $\widetilde{K}$-type $\widetilde{E}_{\mu} = E_{\mu - \rho_c}$ shows up in ${\rm DI}(\aq(\lambda))$.
\end{enumerate}
\end{alg}

\begin{example}
Consider the $A_{\frq}(\lambda)$ module of $U(5,6)$ corresponding to the following $\aq$-chain via \eqref{chain-Aqlambda}:
$$(7,6,5,4,3)^{3,2} \cup (5,4,3,2,1)^{2,3} \cup (1)^{0,1}.$$
This $A_{\frq}(\lambda)$ module is fair but not weakly good. Let us arrange the $\alpha_i$ and $\beta_j$ as follows:
$$
(\alpha_1,\alpha_1,\alpha_1; \alpha_2, \alpha_2| \beta_1, \beta_1; \beta_2, \beta_2, \beta_2; \beta_3).
$$
One computes that $\mathcal{R}_{12} = \{5,4,3\}$, $\mathcal{R}_{23} = \{1\}$.

Now study the possibilities of $(a_{12},b_{12})$, $(a_{23},b_{23})$:
$$\begin{cases}
a_{12} + b_{12} = 3, \\ a_{12} \leq 3, \\ b_{12} \leq 2,
\end{cases}\quad \begin{cases} a_{23} + b_{23} = 1, \\ a_{23} \leq 1, \\ b_{23} \leq 0, \end{cases}\quad \begin{cases} b_{12} + a_{23} \leq 2, \\ a_{12} + b_{23} \leq 3. \end{cases}
$$
The middle system of equations forces $(a_{23},b_{23}) = (1,0)$, and hence the right system of equations gives
$\begin{cases} b_{12} \leq 1, \\ a_{12} \leq 3. \end{cases}$ Thus the possibilities of $(a_{12},b_{12})$ are $(3,0)$ and $(2,1)$ only.

Do the case for $(a_{12},b_{12}) = (3,0)$, $(a_{23},b_{23}) = (1,0)$ first.
So $\mathcal{S}_{12} = \{5,4,3\}$ and $\mathcal{S}_{23} = \{1\}$ is the only possibility, i.e.
$$(5,4,3; \alpha_2, 1| \beta_1, \beta_1; 5,4,3; 1)$$
Now will fill the unrepeated entries $\{7,6\}$ to the unfilled $\alpha_1, \beta_1$ positions, and
$\{2\}$ to the unfilled $\alpha_2, \beta_2$ positions. So the only possibility is
$$(5,4,3; 2, 1| 7,6; 5,4,3; 1)$$

Now do the case $(a_{12},b_{12}) = (2,1)$, $(a_{23},b_{23}) = (1,0)$.
We still have $\mathcal{S}_{23} = \{1\}$ as the only possibility,
but now $\mathcal{S}_{12}$ can be $\{5,4\}$, $\{5,3\}$ and $\{4,3\}$.

Consider the case of $\mathcal{S}_{12} = \{5,4\}$, $\mathcal{S}_{23} = \{1\}$,
i.e.
$$(5,4,\alpha_1; 3, 1| 3, \beta_1; 5,4,\beta_2; 1)$$
The unrepeated entries are $\{7,6\}$ for $\alpha_1,\beta_1$, $\{2\}$ for $\alpha_2,\beta_2$. So we have two possibilities:
$$
(7,5,4; 3, 1| 6,3; 5,4,2; 1), \quad (6,5,4; 3, 1| 7,3; 5,4,2; 1).
$$
Similarly, the two possibilities for $\mathcal{S}_{12} = \{4,3\}$, $\mathcal{S}_{23} = \{1\}$ are
$$
(7,4,3; 5, 1| 6,5; 4,3,2; 1), \quad (6,4,3; 5, 1| 7,5; 4,3,2; 1).
$$
And for $\mathcal{S}_{12} = \{5,3\}$, $\mathcal{S}_{23} = \{1\}$, we have
$$
(7,5,3; 4, 1| 6,4; 5,3,2; 1), \quad (6,5,3; 4, 1| 7,4; 5,3,2; 1).
$$
To sum up,  the terms $E_{(54321|765431)-\tau-\rho_c}$, $E_{(75431|654321)-\tau-\rho_c}$
and $E_{(65431|754321)-\tau-\rho_c}$ all may appear in  ${\rm DI}(\aq(\lambda))$.

To check their multiplicities in the Dirac index, note that our $\aq(\lambda)$ comes from translation
$$
(11,10,9,8,7)^{3,2} \cup (6,5,4,3,2)^{2,3} \cup (1)^{0,1} \longrightarrow (7,6,5,4,3)^{3,2} \cup (5,4,3,2,1)^{2,3} \cup (1)^{0,1}.
$$
That is,
$$
\lambda=(-4, -4, -4, -1, -1 |-4, -4, -1, -1, -1, 0).
$$
So by Theorem \ref{thm-DI-Aqlambda}, the Dirac index of $\aq(\lambda)$ is obtained by
\begin{equation} \label{eq-samesign}
\begin{aligned}
+\widetilde{E}_{(9, 8, 7; 3, 2| 11, 10; 6, 5, 4; 1) - \tau} &\longrightarrow +\widetilde{E}_{(5,4,3; 2,1| 7,6; 5,4,3; 1) - \tau}= + E_{(5, 4, 3, 2, 1| 7, 6, 5, 4, 3, 1) - \tau-\rho_c}\\
-\widetilde{E}_{(11,9,8; 4, 2| 10,7; 6,5,3; 1) - \tau}&\longrightarrow -\widetilde{E}_{(7,5,4; 3, 1| 6,3; 5,4,2; 1) - \tau} = -E_{(7,5,4,3, 1| 6,5,4,3,2,1) - \tau-\rho_c}\\
+\widetilde{E}_{(10,9,8; 4, 2| 11,7; 6,5,3; 1) - \tau}&\longrightarrow +\widetilde{E}_{(6,5,4; 3, 1| 7,3; 5,4,2; 1) - \tau} = +E_{(6,5,4,3, 1| 7,5,4,3,2, 1) - \tau-\rho_c}\\
-\widetilde{E}_{(11,8,7; 6, 2| 10,9; 5,4,3; 1) - \tau}&\longrightarrow -\widetilde{E}_{(7,4,3; 5, 1| 6,5; 4,3,2; 1) - \tau} = -E_{(7,5,4,3, 1| 6,5,4,3,2,1) - \tau-\rho_c}\\
+\widetilde{E}_{(10,8,7; 6, 2| 11,9; 5,4,3; 1) - \tau}&\longrightarrow +\widetilde{E}_{(6,4,3; 5, 1| 7,5; 4,3,2; 1) - \tau} = +E_{(6,5,4,3, 1| 7,5,4,3,2, 1) - \tau-\rho_c}\\
-\widetilde{E}_{(11,9,7; 5, 2| 10,8; 6,4,3; 1) - \tau}&\longrightarrow -\widetilde{E}_{(7,5,3; 4, 1| 6,4; 5,3,2; 1) - \tau} = -E_{(7,5,4,3, 1| 6,5,4,3,2,1) - \tau-\rho_c} \\
+\widetilde{E}_{(10,9,7; 5, 2| 11,8; 6,4,3; 1) - \tau}&\longrightarrow +\widetilde{E}_{(6,5,3; 4, 1| 7,4; 5,3,2; 1) - \tau} = +E_{(6,5,4,3, 1| 7,5,4,3,2, 1) - \tau-\rho_c}
\end{aligned}
\end{equation}
Here the first column are some $\widetilde{E}_{w\rho}$ appearing ${\rm DI}(\aq(0))$, whose signs
are obtained by the determinant of
$w \in W(\mathfrak{l},\mathfrak{t})^1$ sending $\rho = (11,10,9;6,5|8,7;4,3,2;1)$ to $w\rho$.
The second column are the $\widetilde{E}_{\nu}$'s in Step (5) of Algorithm \ref{indexalg}, and
the equality at the last column is given by \eqref{E-tilde-mu}.

As for the other $w \in W(\mathfrak{l},\mathfrak{t})^1$ not listed in \eqref{eq-samesign}, the resulting $\widetilde{E}_{w\rho} \longrightarrow \widetilde{E}_{\nu}$ is singular and hence equal to $0$. Therefore,
$$
{\rm DI}(\aq(\lambda)) = +E_{(5,4,3, 2,1| 7,6, 5,4,3, 1) - \tau - \rho_c} - 3E_{(7,5,4,3, 1| 6,5,4,3,2,1) - \tau - \rho_c } + 3E_{(6,5,4,3,1| 7,5,4,3,2,1) - \tau - \rho_c}.
$$
\qed
\end{example}
Note that in \eqref{eq-samesign}, the sign for any fixed $E_{\mu}$ appearing on the right column is the same. This is
indeed true in general.
\begin{thm}\label{thm-DI-SUpq-Aqlambda}
 Every weakly fair $\aq(\lambda)$ such that \eqref{eq-inequality} is solvable for $(a_{ij},b_{ij})$
has nonzero Dirac index. More precisely, suppose the $\widetilde{K}$-type $E_{\mu-\rho_c}$ is obtained by applying
Algorithm \ref{indexalg} to the choice of solutions $(a_{ij},b_{ij})$ in \eqref{eq-inequality},
then $E_{\mu - \rho_c}$ appears in ${\rm DI}(\aq(\lambda))$ with multiplicity
$$\prod_{\{i < j|\mathcal{R}_{ij} \neq \phi\}} \begin{pmatrix} a_{ij} + b_{ij} \\ a_{ij} \end{pmatrix}.$$
\end{thm}

\begin{example}\label{exam-BP}
The paper \cite{BP15} studied the Dirac cohomology of certain unipotent representations $\pi_u$. When $G = U(p,q)$, these representations correspond to the chains
$$(A_1,\dots,a_1)^{p_1,q_1} \cup (A_2,\dots,a_2)^{p_2,q_2}$$
with $A_1+a_1 \geq A_2 + a_2$ so that the corresponding $\aq(\lambda)$ is in weakly fair range, and:
\begin{itemize}
\item $\{A_1,\dots,a_1\} \subseteq \{A_2,\dots,a_2\}$, $p_1 \leq q_2$, and $q_1 \leq p_2$; or
\item $\{A_1,\dots,a_1\} \supseteq \{A_2,\dots,a_2\}$, $p_1 \geq q_2$, and $q_1 \geq p_2$.
\end{itemize}
We only focus on the first case, which has $\mathcal{R} = \mathcal{R}_{12} = \{A_1,\dots,a_1\}$. Note that since $a_{12} + b_{12} = p_1+q_1$, and $a_{12} \leq p_1$, $b_{12} \leq q_1$, the only option for $(a_{12},b_{12})$ is $(a_{12},b_{12})=(p_1,q_1)$. Therefore, each $E_{\mu-\rho_c}$ appearing in ${\rm DI}(\pi_u)$ has multiplicity $\begin{pmatrix} p_1 + q_1 \\ p_1 \end{pmatrix}$.

On the other hand, Theorem 5.3 of \cite{BP15} implies that each $E_{\nu}$ appearing in the Dirac {\it cohomology} has multiplicity $\begin{pmatrix} p_1 + q_1 \\ p_1 \end{pmatrix}$. This implies that there are no cancellations in the expression of
${\rm DI}(\pi_u) = \pi_u \otimes {\rm Spin}_G^+ - \pi_u \otimes {\rm Spin}_G^-$.\hfill\qed
\end{example}

\begin{remark}\label{rmk-thm-DI-SUpq-Aqlambda}
Conjecture 10.3 of \cite{H15} suggests that
there should be no cancellation among the $\widetilde{K}$-types when passing
from $H_D(\pi)$ to ${\rm DI}(\pi)$. We have just seen this is the case for the unipotent
representations in Example \ref{exam-BP}. Moreover, our calculation on \texttt{atlas} implies
that the conjecture  holds on small rank groups of $U(p, q)$.
\end{remark}

\begin{proof}
Fix a $\widetilde{K}$-type $E_{\mu - \rho_c}$ which is obtained as in Algorithm \ref{indexalg}, then
by the construction therein, it must come from different subsets of $\mathcal{S}_{ij}$
and $\mathcal{R}_{ij} \backslash \mathcal{S}_{ij}$ for a {\bf fixed choice of $(a_{ij},b_{ij})$}
with $\#\mathcal{S}_{ij} = a_{ij}$, $\#(\mathcal{R}_{ij} \backslash \mathcal{S}_{ij}) = b_{ij}$.

For any $\mathcal{S}_{ij}$ with $\#\mathcal{S}_{ij}= a_{ij}$ ($\begin{pmatrix} a_{ij} + b_{ij} \\ a_{ij} \end{pmatrix}$ choices in total),
there is a unique way to fill in the unrepeated entries $\alpha_i, \alpha_j, \beta_i, \beta_j$ so that
Algorithm \ref{indexalg} gives $E_{\mu - \rho_c}$. Therefore, as in \eqref{eq-samesign}, there is a total of
$\displaystyle \prod_{\{i < j\ |\ \mathcal{R}_{ij} \neq  \phi\}} \begin{pmatrix} a_{ij} + b_{ij} \\ a_{ij} \end{pmatrix}$
copies of $\pm E_{\mu - \rho_c}$ appearing in the Dirac index of $\aq(\lambda)$.
So it suffices to show that for each choice of $\mathcal{S}_{ij}$, the algorithm
gives a copy of $E_{\mu - \rho_c}$ \textbf{of the same sign}.

\medskip

We focus on a fixed $i < j$ with $\mathcal{R}_{ij} \neq \phi$. Let
$r := \#\mathcal{R}_{ij}$, $a := a_{ij}$, $b := b_{ij}$ so that $a+b = r$.
Note that by the definition of $\mathcal{R}_{ij}$ \eqref{rij},
the elements $s_1 > \dots > s_r$ in $\mathcal{R}_{ij}$ are consecutive
integers. Therefore, for any choice of
$$\{\phi_1 > \dots > \phi_a \} = \mathcal{S}_{ij}, \quad \{\psi_1 > \dots > \psi_b \} = \mathcal{R}_{ij} \backslash \mathcal{S}_{ij},$$
they must appear in the form
\begin{equation} \label{zeta}
\begin{aligned}
\nu(\mathcal{S}_{ij}) := (\dots, \overbrace{\alpha_i, \dots, \alpha_i, \phi_1 >\dots > \phi_a, \alpha_i', \dots, \alpha_i'}^{p_i}, \dots, \overbrace{\alpha_j, \dots, \alpha_j, \psi_1 > \dots > \psi_b, \alpha_j', \dots, \alpha_j'}^{p_j}, \dots | \\
\dots, \overbrace{\beta_i, \dots, \beta_i, \psi_1 > \dots > \psi_b, \beta_i', \dots, \beta_i'}^{q_i}, \cdots, \overbrace{\beta_j, \dots, \beta_j, \phi_1 > \dots > \phi_a, \beta_j', \dots, \beta_j'}^{q_j}, \dots ),
\end{aligned}
\end{equation}
where the amount of $\alpha_i, \alpha_i', \alpha_j, \alpha_j', \beta_i, \beta_i', \beta_j, \beta_j'$ entries
are equal for all choices of $\mathcal{S}_{ij}$, whose values are determined by our fixed choice of $E_{\mu - \rho_c}$.

\smallskip

Our proof will follow from the following claims.

\medskip

\noindent \underline{\bf Claim 1}\ \emph{For each $\mathcal{S}_{ij} \subset \mathcal{R}_{ij}$, let
$w(\mathcal{S}_{ij}) \in W(\mathfrak{l},\mathfrak{t})^1$ be such that the translation principle gives
$\widetilde{E}_{w(\mathcal{S}_{ij})\rho} \longrightarrow \widetilde{E}_{\nu(\mathcal{S}_{ij})}$.
Then $\det(w(\mathcal{S}_{ij}))$ is of the same sign for all $\mathcal{S}_{ij}$ with $\#\mathcal{S}_{ij} = a$.}

\noindent {\bf Proof of Claim 1.}\ Suppose $\mathcal{S}_{ij}' = \{\kappa_1 >\dots > \kappa_a\}$ and $\mathcal{R}_{ij} \backslash \mathcal{S}_{ij}' = \{\eta_1 > \dots > \eta_b\}$ be such that
\begin{equation} \label{zetaprime}
\begin{aligned}
\nu(\mathcal{S}_{ij}') = (\dots, \overbrace{\alpha_i, \dots, \alpha_i, \kappa_1 >\dots > \kappa_a, \alpha_i', \dots, \alpha_i'}^{p_i}, \dots, \overbrace{\alpha_j, \dots, \alpha_j, \eta_1 > \dots > \eta_b, \alpha_j', \dots, \alpha_j'}^{p_j}, \cdots | \\
\dots, \overbrace{\beta_i, \dots, \beta_i, \eta_1 > \dots > \eta_b, \beta_i', \dots, \beta_i'}^{q_i}, \dots, \overbrace{\beta_j,
\dots, \beta_j, \kappa_1 > \dots > \kappa_a, \beta_j', \dots, \beta_j'}^{q_j}, \dots ),
\end{aligned}
\end{equation}
then one has
$$\{\kappa_1 >\dots > \kappa_a,\eta_1 > \dots > \eta_b\} = \{\phi_1 > \dots > \phi_a, \psi_1 > \dots > \psi_b\} = \{s_1 > \dots > s_r\}.$$
So $w(\mathcal{S}_{ij})$ and $w(\mathcal{S}'_{ij})$ differ by a Weyl group element in $S_{p_i + q_i} \times S_{p_j + q_j}$ of the form
\begin{align*}
&(\overbrace{*,\cdots,*,\phi_1, \dots, \phi_a,*,\cdots,*}^{p_i},\overbrace{*,\cdots,*,\psi_1, \dots, \psi_b,*,\cdots,*}^{q_i}) \\
\rightarrow\ & (\overbrace{*,\cdots,*,\kappa_1, \dots, \kappa_a,*,\cdots,*}^{p_i},\overbrace{*,\cdots,*,\eta_1, \dots, \eta_b,*,\cdots,*}^{q_i}) \in S_{p_i+q_i}
\end{align*}
\begin{align*}
&(\overbrace{*,\cdots,*,\psi_1, \dots, \psi_b,*,\cdots,*}^{p_j},\overbrace{*,\cdots,*,\phi_1, \dots, \phi_a,*,\cdots,*}^{q_j}) \\
\rightarrow\ & (\overbrace{*,\cdots,*,\eta_1, \dots, \eta_b,*,\cdots,*}^{p_j},\overbrace{*,\cdots,*,\kappa_1, \dots, \kappa_a,*,\cdots,*}^{q_j})
\in S_{p_j+q_j}
\end{align*}
where the entries marked with $*$ are unchanged. It is not hard to see this Weyl group element has determinant $1$. Therefore,
one concludes that $\det(w(\mathcal{S}_{ij})) = \det(w(\mathcal{S}'_{ij}))$ for all $\#\mathcal{S}_{ij} = \#\mathcal{S}'_{ij} = a$.


\medskip

\noindent \underline{\bf Claim 2}\ \emph{The Weyl group elements in $S_p \times S_q$ translating $\nu(\mathcal{S}_{ij})$, $\nu(\mathcal{S}_{ij}')$ to $\mu$ have the same determinant.}

\noindent {\bf Proof of Claim 2.}\
Firstly, we translate $\nu(\mathcal{S}_{ij})$ and $\nu(\mathcal{S}_{ij}')$ by the same element in $S_p \times S_q$ such that
the $\phi, \psi$-entries and $\kappa, \eta$-entries of $\nu(\mathcal{S}_{ij})$ and $\nu(\mathcal{S}_{ij}')$ respectively appears at the
first $r = a+b$ coordinates. So we are left to check the Weyl group element in $S_{a+b} \times S_{a+b}$
$$(\phi_1,\dots,\phi_a,\psi_1,\dots,\psi_b | \psi_1,\dots,\psi_b, \phi_1,\dots,\phi_a) \to
(\kappa_1,\dots,\kappa_a,\eta_1,\dots,\eta_b | \eta_1,\dots,\eta_b, \kappa_1,\dots,\kappa_a)$$
has determinant $1$. Indeed, both elements above can be mapped to the dominant element $(s_1,\dots, s_r | s_1,\dots,s_r)$ by
\begin{align*}
&(\phi_1,\dots,\phi_a,\psi_1,\dots,\psi_b | \psi_1,\dots,\psi_b, \phi_1,\dots,\phi_a) \\ \stackrel{(e,w_{b,a})}{\longrightarrow}
&(\phi_1,\dots,\phi_a,\psi_1,\dots,\psi_b | \phi_1,\dots,\phi_a,\psi_1,\dots,\psi_b) \\
\stackrel{(w_{\phi,\psi},w_{\phi,\psi})}{\longrightarrow} &(s_1, \dots, s_r | s_1, \dots, s_r),
\end{align*}
where $\det(e,w_{b, a}) = (-1)^{a b}$ and $\det(w_{\phi,\psi},w_{\phi,\psi}) = \det(w_{\phi,\psi})^2 = 1$.
Replacing $w_{\phi,\psi}$ with $w_{\kappa,\eta}$ above, the argument is identical for $(\kappa_1,\dots,\kappa_a,\eta_1,\dots,\eta_b | \eta_1,\dots,\eta_b, \kappa_1,\dots,\kappa_a)$,
both of which involve a Weyl group element of determinant $(-1)^{a \dot b}$. Hence the claim follows.

\medskip

To conclude, for a fixed $\widetilde{K}$-type $\widetilde{E}_{\mu} = E_{\mu - \rho_c}$,
then each $w(\mathcal{S}_{ij})$ (with $\#\mathcal{S}_{ij}= a$) on the first column
of Equation \eqref{eq-samesign} has the same sign by Claim 1.
Also, the equality on the last column of \eqref{eq-samesign} makes
the same sign change thanks to Claim 2. Consequently,
each $E_{\mu-\rho_c}$ in \eqref{eq-samesign} contributes to the
Dirac index of $\aq(\lambda)$ with the same sign, and the result follows.
\end{proof}

\section{Vogan's conjecture on $\widehat{U(p, q)}$}\label{sec-Vogan-conj}

In this section, we provide some calculations on \texttt{atlas} relevant to Conjecture \ref{conj-Vogan} and Conjecture \ref{conj-diracseries}.
\begin{example}\label{exam-su33}
Let us single out all the irreducible representations of $SU(3, 3)$ with infinitesimal character $\rho_c$ ($=$ \texttt{[2, 2, 1, 1, 0]} in \texttt{atlas}) and full support as follows. Here some outputs are omitted for convenience.
\begin{verbatim}
G:SU(3,3)
set all=all_parameters_gamma(G, [2, 2, 1, 1, 0])
#all
Value: 38
for p in all do if #support(x(p))=5 and is_unitary(p) then prints(p) fi od
final parameter(x=210,lambda=[4,5,3,2,1]/1,nu=[2,2,1,1,0]/1)
final parameter(x=205,lambda=[4,4,2,2,0]/1,nu=[2,2,1,1,0]/1)
final parameter(x=204,lambda=[4,4,2,2,0]/1,nu=[2,2,1,1,0]/1)
\end{verbatim}
One can check that the last two parameters can be obtained by
restricting the $\aq(\lambda)$-modules in $U(3,3)$ with $\aq$-chains
$$\left(3,2,1\right)^{1,2} \cup \left(3,2,1\right)^{2,1}, \quad \left(3,2,1\right)^{2,1} \cup \left(3,2,1\right)^{1,2}$$
to $SU(3,3)$.

We note that the first parameter does not correspond to a weakly fair $\aq(\lambda)$. Indeed,
by checking its associated variety by \texttt{atlas} \cite{AV,At}, it is a union of four $K$-nilpotent
orbits in $\mathfrak{p}$. Since all nonzero $\aq(\lambda)$-modules in $U(p,q)$ has associated variety
equal to the closure of a single $K$-nilpotent orbit (c.f. \cite{T2}), this cannot be an $\aq(\lambda)$-module
and hence this gives a counter-example of Conjecture \ref{conj-Vogan}. \hfill\qed
\end{example}

To account for the first module in the previous example, let $G = SU(p,p)$ and
$$\Sigma_p := {\rm Ind}_{P}^{G}(\sigma^{p}), \quad \Sigma_p' := {\rm Ind}_{P}^{G}(\sigma^{p+1}),$$
 where $P = MAN$ is the Siegel parabolic of $G$, and $\sigma$
is the nontrivial character of the component group of $M$ ($\cong \mathbb{Z}/2\mathbb{Z}$).
The associated variety of $\Sigma_p$ and $\Sigma_p'$ are both equal to the union of all $K$-real forms
of the complex nilpotent orbit corresponding to the partition $[2^p]$ ($p+1$ orbits in total).

By \cite{MT}, $\Sigma_p$ is reducible for all $p$,
whose composition factors consist of $\aq(\lambda)$-modules in $U(p,p)$ of the form
$$\left(p,\dots,1\right)^{i,p-i} \cup \left(p,\dots,1\right)^{p-i,i}, \quad 0 \leq i \leq p$$
restricted to $G$, all appearing with multiplicity one.
On the other hand, \cite[Section 5]{J} implies that $\Sigma_p'$
is irreducible, and the first parameter in Example \ref{exam-su33}
corresponds to $\Sigma_3'$. This can be verified by \texttt{atlas} as follows:
\begin{verbatim}
G:SU(3,3)
set P=KGP(G,[0,1,3,4])[9]
set L=Levi(P)
L
Value: disconnected quasisplit real group with Lie algebra 'sl(3,C).gl(1,R)'
real_induce_irreducible(trivial(L),G)
Value:
1*parameter(x=210,lambda=[4,5,3,2,1]/1,nu=[2,2,1,1,0]/1) [0]
\end{verbatim}

To study the Dirac cohomology of $\Sigma_p'$, note that its infinitesimal character is conjugate to $\rho_c$,
and its $K$-types are of the form
\begin{equation} \label{eq-sigmaktypes}
\Sigma_p'|_K = \begin{cases}
\displaystyle\bigoplus_{a_i \in \mathbb{Z},\ a_1 \geq \dots \geq a_p} E_{(a_1, \dots, a_p |-a_p, \dots, -a_1)} & \text{if}\ p\ \text{is odd} \\
\displaystyle\bigoplus_{a_i \in \mathbb{Z},\ a_1 \geq \dots \geq a_p} E_{(a_1 + \frac{1}{2}, \dots, a_p + \frac{1}{2}| -a_p - \frac{1}{2}, \dots, -a_1- \frac{1}{2})} & \text{if}\ p\ \text{is even}
\end{cases}
\end{equation}
by Frobenius reciprocity. Therefore, by Theorem \ref{thm-HP}, if $H_D(\Sigma_p') \neq 0$, then
there must be a $K$-type $E_{\mu}$ in $\Sigma_p'$, and a subset $\langle \Phi \rangle \subset \Delta^+(\mathfrak{p},\mathfrak{t})$ such that
$$\mu - \rho_n + \langle \Phi \rangle = w\rho_c - \rho_c$$
is the highest weight of a $\widetilde{K}$-type appearing in $H_D(\Sigma_p')$. This forces $w\rho_c=\rho_c$
and
\begin{equation} \label{eq-p-parity}
\mu + \langle \Phi \rangle = \rho_n.
\end{equation}
Note that $\rho_n$ is half the sum of all roots of the form $\epsilon_i - \epsilon_j$ with $1 \leq i \leq p < j \leq 2p$, which has $p^2$ roots in total. Therefore, if $p$ is odd, then the coordinates of $\rho_n$ consists of half-integers. However, by the choice of $\mu$ given in \eqref{eq-sigmaktypes}, the coordinates of left hand side of \eqref{eq-p-parity} must be integers. Therefore, $\Sigma_p'$ do not appear in the Dirac series for odd $p$ and similarly for even $p$. In other words, $\Sigma_p'$ does not violate Conjecture \ref{conj-diracseries}.

\bigskip

\centerline{\scshape Funding}
Dong was supported by the National Natural Science Foundation of China (grant 11571097, 2016-2019). Wong is supported by the National Natural Science Foundation of China (grant 11901491) and the Presidential Fund of CUHK(SZ).

\medskip
\centerline{\scshape Acknowledgements}
We thank Prof. Matumoto and Prof. Trapa sincerely for suggesting \cite{J} and \cite{MT} to the authors, along with many very helpful discussions.

\end{document}